\documentclass[a4paper,12pt]{amsart}
\usepackage{amssymb}
\usepackage[colorlinks=true]{hyperref}
\usepackage{tikz}
\usetikzlibrary{arrows}
\allowdisplaybreaks[1]
\title{Jucys--Murphy elements and Weingarten matrices}
\author[P.~Zinn-Justin]{Paul~Zinn-Justin}
\address{Paul Zinn-Justin, 
UPMC Univ Paris 6, CNRS UMR 7589, LPTHE, 75252 Paris Cedex, FRANCE.}
\email{pzinn\,@\,lpthe.jussieu.fr}
\thanks{PZJ was supported by EU Marie Curie Research Training Network
``ENRAGE'' MRTN-CT-2004-005616, ESF program ``MISGAM''
and ANR program ``GRANMA'' BLAN08-1-13695.}
\numberwithin{equation}{section}
\date{July 2009}
\newtheorem{prop}{Proposition} 
\newtheorem{lemma}{Lemma}
\newcommand\vcenterbox[1]{\vcenter{\hbox{#1}}}
\newcommand\C{\mathbb{C}}
\renewcommand\S{\mathcal{S}}
\renewcommand\H{\mathcal{H}}
\newcommand\B{\mathcal{B}}
\newcommand\Z{\mathbb{Z}}
\newcommand\transp[2]{s_{#1,#2}{}}
\newcommand\ket[1]{\left|#1\right\rangle}
\newcommand\bra[1]{\left\langle#1\right|}
\newcommand\braket[2]{\left\langle#1\,\vrule\,#2\right\rangle}
\newdimen{\cellsize}
\newcommand\bigboxes{\setlength{\cellsize}{18pt}\def\boxformat{}}
\newcommand\medboxes{\setlength{\cellsize}{14pt}\def\boxformat{}}
\newcommand\smallboxes{\setlength{\cellsize}{8pt}\def\boxformat{\scriptstyle}}
\medboxes
\newsavebox{\cellcontent}
\def\hidehrule#1#2{\kern-#1
  \hrule height#1 depth#2 \kern-#2 }%
\def\hidevrule#1#2{\kern-#1{\dimen\cellcontent=#1%
    \advance\dimen\cellcontent by#2\vrule width\dimen\cellcontent}\kern-#2 }%
\def\makeblankbox#1#2{\hbox{\lower\dp\cellcontent\vbox{\hidehrule{#1}{#2}%
    \kern-#1 
    \hbox to \wd\cellcontent{\hidevrule{#1}{#2}%
      \raise\ht\cellcontent\vbox to #1{}
      \lower\dp\cellcontent\vtop to #1{}
      \hfil\hidevrule{#2}{#1}}%
    \kern-#1\hidehrule{#2}{#1}}}}
\newcommand\cellify[1]{\defaultcell%
\sbox{\cellcontent}{\vbox to \cellsize{%
\vfill%
\hbox to \cellsize{\hfill$\boxformat #1$\hfill}
\vfill}}%
\rlap{\drawnbox}
\usebox{\cellcontent}}
\newcommand\tableau[1]{\vtop{\let\\\cr
\baselineskip -16000pt \lineskiplimit 16000pt \lineskip 0pt
\ialign{&\cellify{##}\cr#1\crcr}}}
\newcommand\defaultcell{\gdef\drawnbox{
\makeblankbox{0.2pt}{0.2pt}
}}
\newcommand\vdotscell{\gdef\drawnbox{\kern-1.6pt\vbox{\baselineskip=4pt\lineskiplimit=0pt\hbox{}\hbox{.}\hbox{.}\hbox{.}\hbox{}}}}
\newcommand\hdotscell{\gdef\drawnbox{\vbox to \cellsize{\hbox{\kern1pt$\ldotp\ldotp\ldotp$}}}}
\newcommand\vhdotscell{\gdef\drawnbox{\rlap{\kern-1.6pt\vbox{\baselineskip=4pt\lineskiplimit=0pt\hbox{}\hbox{.}\hbox{.}\hbox{.}\hbox{}}}\vbox to \cellsize{\hbox{\kern1pt$\ldotp\ldotp\ldotp$}}}}
\begin{document}
\begin{abstract}
We provide a compact proof of the recent formula of Collins and Matsumoto
for the Weingarten matrix of the orthogonal group using Jucys--Murphy
elements.
\end{abstract}
\maketitle
\begin{center}
\end{center}
\section{Introduction}
\subsection{Motivation}
In this note, we discuss the calculation of the so-called Weingarten matrix
for unitary and orthogonal groups.
Let us recall its origin. The initial problem is the computation
of the integral over certain compacts groups of matrices,
of polynomials of the entries of these matrices. 
More abstractly, the problem amounts
to finding the projector from a tensor product of representations
onto its invariant subspace. As will be described in the next section,
the result is expressed in terms of a certain matrix $\mathbf{W}$,
the Weingarten matrix.
Such computations are useful in various areas of theoretical physics
and mathematics: it seems that they were originally considered by
Weingarten in the context of gauge theories \cite{Weingarten}
and then reappeared in matrix models and in free probability.
In particular Collins et al \cite{CS-weingarten,Collins-weingarten} 
studied this Weingarten matrix
using representation theory of the symmetric group and derived some explicit
expressions for $\mathbf{W}$. It is the goal of this note to provide elementary
proofs of these results using Jucys--Murphy elements. Note that
the only case where this procedure is new is the case of the orthogonal
group, where the original proof by Collins and Matsumoto \cite{CM-WOn} involved 
fairly involved machinery (see also \cite{Banica-WOn}); use of Jucys--Murphy for the unitary case
had already been considered in \cite{Novak-weingarten,MN-weingarten}.\footnote{The author thanks B.~Collins for pointing out these references.}

\subsection{Weingarten matrix as a pseudo-inverse}
Let $V$ be a (real, or complex) vector 
space endowed with a symmetric non-degenerate bilinear form $\braket{\cdot}{\cdot}$.
We are interested in an explicit expression for the orthogonal projector $\Pi$ 
onto a given subspace $V_0$.
The strategy is then to (i) identify a set of elements $\ket{u_i}\in V$ which generate $V_0$;
and (ii) try the Ansatz $\Pi=\sum_{i,j} \ket{u_i} \mathbf{W}_{i,j} \bra{u_j}$
where the matrix $\mathbf{W}$ is to be determined.
For $\Pi$ to be orthogonal, $\mathbf{W}$ must be a symmetric matrix.
Furthermore, writing that $\bra{u_i}\Pi\ket{u_j}=\braket{u_i}{u_j}$
leads to the following matrix identity: $\mathbf{GWG}=\mathbf{G}$,
where $\mathbf{G}$ is the {\em Gram matrix}: $\mathbf{G}_{i,j}=\braket{u_i}{u_j}$.
In fact it is easy to see that these are all the conditions on $\mathbf{W}$.
If the $\ket{u_i}$ are independent, $\mathbf{G}$ is invertible and we conclude directly 
that $\mathbf{W}=\mathbf{G}^{-1}$.
If $\mathbf{G}$ is not invertible, $\mathbf{W}$ is not uniquely determined, but
a convenient choice \cite{CS-weingarten}, which we make here,
is to require $\mathbf{W}$ to be the {\em pseudo-inverse}\/ of $\mathbf{G}$,
that is the symmetric matrix such that $\mathbf{GWG}=\mathbf{G}$ and 
$\mathbf{WGW}=\mathbf{W}$.

In what follows, $V$ will be equipped with a representation of some group,
and $V_0$ will be the invariant subspace with respect to that action.
In this case $\mathbf{W}$ will be called the
{\em Weingarten matrix}.

\subsection{Jucys--Murphy elements}\label{secmurphy}
In this note we only use standard methods due to Young, 
Jucys \cite{Jucys} and Murphy \cite{Murphy}.
Let $n$ be a positive integer. We consider in what follows the natural embedding
$\C[\S_1]\subset\C[\S_2]\subset\cdots\subset\C[\S_n]$, where elements
of $\C[\S_k]$ act trivially on numbers greater than $k$.
Denote the transposition of $i$ and $j$ by $\transp{i}{j}$.
We recall that to a partition $\lambda$
with $|\lambda|=n$ parts, represented by a Young diagram,
one can associate the set $SYT(\lambda)$
of standard Young tableaux $T$ of shape $\lambda$ 
i.e.\ fillings of each box $(i,j)\in\lambda$
with numbers $T(i,j)\in\{1,\ldots,n\}$ which are increasing along rows
and columns.

We now define the Jucys--Murphy elements. $m_k$ is an element
of $\C[\S_k]$ given by:
\begin{equation}\label{defmurphy}
m_1=0,\qquad m_k=\sum_{i=1}^{k-1} \transp{i}{k}\qquad k=2,\ldots,n
\end{equation}
(the definition of $m_1$ is a convenient convention). 
Note that $m_k$ ($k\ge2$) commutes with
$\C[\S_{k-1}]$; this implies in particular that the $m_k$, $k=2,\ldots,n$,
form a commutative subalgebra of $\C[\S_n]$. It is in fact a maximal
commutative subalgebra.

Next we introduce 
Young's {\it orthogonal}\/ idempotents $e_T$, $T$ standard Young tableau
with $n$ boxes. They are characterized by the following properties:
first they are a complete set of orthogonal idempotents:
\begin{equation}\label{idem}
e_T e_{T'}=\delta_{TT'} e_T\qquad \sum_{T: |T|=n} e_T = 1
\end{equation}
and secondly they ``diagonalize'' the Jucys--Murphy elements:
\begin{equation}\label{diagmurphy}
m_k e_T = e_T m_k = c(T_k) e_T\qquad k=1,\ldots,n
\end{equation}
where $c(T_k)$ is the content of the box labelled $k$ in $T$: $c(T_k)=j-i$ if $T(i,j)=k$.

{\em Example:} there are two tableaux of shape $\tableau{&\\\\}$:
\newcommand\taba{e_{\smallboxes\tableau{1&2\\3\\}}}
\newcommand\tabb{e_{\smallboxes\tableau{1&3\\2\\}}}
\begin{align*}
m_1\taba&=0 & m_2\taba&=\taba & m_3\taba&=-\taba\\
m_1\tabb&=0 & m_2\tabb&=-\tabb & m_3\tabb&=\tabb\\
\end{align*}

Finally define 
\begin{equation}
P_\lambda=\sum_{T\in SYT(\lambda)} e_{T}
\end{equation}
$P_\lambda$ is the central idempotent
associated to $\lambda\vdash n$.
It is well-known that $P_\lambda$ can also be expressed in terms
of the corresponding character $\chi_\lambda$ of the symmetric group, namely that
\begin{equation}\label{char}
P_\lambda=\frac{\chi_\lambda(1)}{|\S_n|}\sum_{\sigma\in\S_n} \chi_\lambda(\sigma^{-1})\sigma
\end{equation}
though we shall not need this expression here.
For the sake of completeness, 
let us also remark that since $e_T$ commutes with all the $m_k$, 
it should be expressible as a polynomial of them; and indeed, 
we have the following (Lagrange interpolation type) inductive definition:
$e_T=e_{\bar T} 
\prod_{T': \bar T'=\bar T, T'\ne T} \frac{m_n - c(T'_n)}{c(T_n)-c(T'_n)}$,
where
$\bar T$ denotes the tableau
$T$ with its last box removed. 

\section{Weingarten matrix for the unitary group}
Let $\tau$ be a positive integer.
As warming up, we consider the simple and well-studied case 
\cite{Collins-weingarten,Novak-weingarten,MN-weingarten}
of the action of
$U(\tau)$ (or its complexification $GL(\tau)$)
on $V=(\C^\tau)^{\otimes n}\otimes (\C^\tau)^{\otimes n}\simeq Hom((\C^\tau)^{\otimes n})$
where the first $n$ factors of the tensor product
are in the fundamental representation of
$U(\tau)$ and the last $n$ in its dual representation.
Schur--Weyl duality provides us with a set of invariants indexed
by permutations in $\S_n$.
The corresponding Gram matrix is easily computed:
\[
\mathbf{G}_{\sigma,\sigma'}=\tau^{\text{number of cycles of $\sigma^{-1}\sigma'$}}
\qquad \sigma,\sigma'\in \S_n
\]

Next we need the classical identity: \cite{Jucys}
\begin{prop}[Jucys] \label{uprop}
\begin{equation}\label{uid}
\prod_{k=1}^n (\tau+m_k)=
\sum_{\sigma\in\S_n} \sigma\ \tau^{\text{number of cycles of $\sigma$}}
\end{equation}
\end{prop}
There are many ways to prove this identity. We provide an elementary
one now, based on a standard inductive construction of permutations.
\begin{proof}First, note that there are as many terms in the r.h.s.\ as in
the l.h.s.\ after expanding the product $(\tau+m_1)\cdots(\tau+m_n)$.
Our proof will consist in
an identification term by term. We proceed by induction on $n$.
Consider a permutation $\sigma\in\S_n$.
Either (i) $n$ is a fixed point of $\sigma$, in which
case we can apply
the induction hypothesis to
$\sigma_{|\{1,\ldots,n-1\}}\in\S_{n-1}$
and we know that it corresponds
to one term in $(\tau+m_1)\cdots(\tau+m_{n-1})$
with coefficient $\tau^{\text{number of cycles of $\sigma_{|\{1,\ldots,n-1\}}$}}$; 
furthermore, $\sigma$ has one more cycle that
$\sigma_{|\{1,\ldots,n-1\}}$ and therefore we identify $\sigma$
with the term in $(\tau+m_1)\cdots(\tau+m_n)$ with the same
choice in the first $n-1$ factors, and in which we further pick
the multiplication by $\tau$ inside $\tau+m_n$;
or (ii) $n$ is not a fixed point in which case
we can similarly apply the induction hypothesis to
$\sigma\transp{\sigma^{-1}(n)}{n}{}_{|\{1,\ldots,n-1\}}\in\S_{n-1}$:
\[
\sigma=\text{(other cycles\dots)}
\quad
\vcenterbox{\begin{tikzpicture}[bend right]
\node (a) at (0:1) {$n$};
\node (b) at (120:1) {$\sigma(n)$};
\node (c) at (240:1) {$\sigma^{-1}(n)$};
\draw[->] (a) to (b);
\draw[dotted,->] (b) to (c);
\draw[->] (c) to (a);
\end{tikzpicture}}
\hskip1cm
\sigma\transp{\sigma^{-1}(n)}{n}= \text{(other cycles\dots)}
\quad
\vcenterbox{\begin{tikzpicture}[bend right]
\node (a) at (0:1) {$n$};
\node (b) at (120:1) {$\sigma(n)$};
\node (c) at (240:1) {$\sigma^{-1}(n)$};
\draw[->] (c) to (b);
\draw[dotted,->] (b) to (c);
\draw[->] (a.south east) to[in=305,out=235,min distance=1cm] (a.north east);
\end{tikzpicture}}
\]
Noting that $\sigma$ has as many cycles
as $\sigma\transp{\sigma^{-1}(n)}{n}{}_{|\{1,\ldots,n-1\}}$,
we conclude that it is identified
with the term in $(\tau+m_1)\cdots(\tau+m_n)$ with the same choice
in the first $n-1$ factors, and in which we pick
the transposition $\transp{\sigma^{-1}(n)}{n}$ inside $\tau+m_n$.

The induction is concluded by noting that the statement at $n=1$ is trivial.
\end{proof}

One observes that the r.h.s. of \eqref{uid} 
looks very similar to the entries of
the Gram matrix $\mathbf{G}$. In fact, if we call
$G$ the quantity in \eqref{uid}, it is easy to see that
$\mathbf{G}$ is the matrix of $G\in\C[\S_n]$
acting in either the left or right
regular representation of $\C[\S_n]$, with standard basis $\S_n$.

Now, by inserting $1=\sum_{T:|T|=n} e_T$ in the formula \eqref{uid} above
and applying \eqref{diagmurphy}, one finds that
$e_T G$ depends on $T$ only via its shape $\lambda$; therefore,
one can
write 
\[
G=\sum_{\lambda\vdash n} c_\lambda
P_\lambda\qquad c_\lambda:= \prod_{(i,j)\in \lambda} (\tau+j-i)
\]
Thus, 
$G^{-1}=\sum_{\lambda\vdash n} c_\lambda^{-1} P_\lambda$
when $G$ is invertible, and more generally, noting
that the matrix of $P_\lambda$ acting in left and right representation
is symmetric (cf \eqref{char}), we conclude:
\begin{prop}[Collins] If one defines
\[
W=\sum_{\substack{\lambda\vdash n\\c_\lambda\ne 0}}
c_\lambda^{-1}P_\lambda
\]
then the Weingarten matrix $\mathbf{W}$ is the matrix of $W$
in both left or right regular representation.
\end{prop}
Theorem 2.1 of \cite{Collins-weingarten} is recovered by replacing $P_\lambda$
with its expression \eqref{char} and by noting that $c_\lambda$ is up to a constant factor
 the dimension of the $GL(\tau)$ irreducible representation associated to $\lambda$,
that is the Schur function with partition $\lambda$ and parameters $\underbrace{1,\ldots,1}_\tau$.

\section{Weingarten matrix for the orthogonal group}
We now consider the case of $O(\tau)$ (either real or complex orthogonal group)
acting on $V=(\C^\tau)^{\otimes 2n}\simeq Hom((\C^\tau)^{\otimes n})$
where $\C^\tau$ is in the fundamental representation of
$O(\tau)$. The Weingarten matrix in this case was considered in \cite{CS-weingarten,JBZ-On,CM-WOn}.
Once again Schur--Weyl duality provides us with
a set of generators of the invariant subspace, as elements of the Brauer algebra
of size $n$, which for our purposes are conveniently defined as follows:
they are involutions without fixed points of $\{1,\ldots,2n\}$. Let us denote
by $\B_n$ their set. The Gram matrix turns out to be:
\[
\mathbf{G}_{\pi,\pi'}=
\tau^{\frac{1}{2}\text{number of cycles of $\pi\pi'$}}
\qquad
\pi,\pi'\in \B_n
\]
Graphically, one half of the number of cycles of $\pi\pi'$ is simply
the number of loops produced by pasting together $\pi$ and $\pi'$
viewed as pairing of points $\{1,\ldots,2n\}$:
\[
\pi=\vcenterbox{
\begin{tikzpicture}[x=0.5cm,y=0.5cm]
\path (0,-1) rectangle (7,1);
\draw (0,0) -- (7,0);
\draw[thick] (1,0) to[in=90,out=90] (2,0);
\draw[thick] (3,0) to[in=90,out=90] (5,0);
\draw[thick] (4,0) to[in=90,out=90] (6,0);
\foreach\i in {1,2,...,6}
\path[fill=red,draw=black] (\i,0) circle (2pt) node[below,black] {$\scriptstyle\i$};
\end{tikzpicture}}
\quad
\pi'=\vcenterbox{
\begin{tikzpicture}[x=0.5cm,y=0.5cm]
\path (0,-1) rectangle (7,1);
\draw (0,0) -- (7,0);
\draw[thick] (1,0) to[in=90,out=90] (2,0);
\draw[thick] (3,0) to[in=90,out=90] (6,0);
\draw[thick] (4,0) to[in=90,out=90] (5,0);
\foreach\i in {1,2,...,6}
\path[fill=red,draw=black] (\i,0) circle (2pt) node[below,black] {$\scriptstyle\i$};
\end{tikzpicture}}
\quad
\mathbf{G}_{\pi,\pi'}=
\vcenterbox{\begin{tikzpicture}[x=0.5cm,y=0.5cm]
\path (0,-1) rectangle (7,1);
\draw (0,0) -- (7,0);
\draw[thick] (1,0) to[in=270,out=270] (2,0);
\draw[thick] (3,0) to[in=270,out=270] (5,0);
\draw[thick] (4,0) to[in=270,out=270] (6,0);
\draw[thick] (1,0) to[in=90,out=90] (2,0);
\draw[thick] (3,0) to[in=90,out=90] (6,0);
\draw[thick] (4,0) to[in=90,out=90] (5,0);
\foreach\i in {1,2,...,6}
\path[fill=red,draw=black] (\i,0) circle (2pt);
\end{tikzpicture}}=\tau^2
\]

$\S_{2n}$ acts by conjugation on $\B_n$, and this action is transitive.
Let us pick a particular element $\beta_n\in \B_n$:
\[
\beta_n=\transp{1}{2}\cdots\transp{2n-1}{2n}
=\vcenterbox{
\begin{tikzpicture}[x=0.5cm,y=0.5cm]
\path (0,-1) rectangle (9,1);
\draw (0,0) -- (4.5,0);
\draw[dotted] (4.5,0) -- (6.5,0);
\draw (6.5,0) -- (9,0);
\draw[thick] (1,0) to[in=90,out=90] (2,0);
\draw[thick] (3,0) to[in=90,out=90] (4,0);
\draw[thick] (7,0) to[in=90,out=90] (8,0);
\foreach\i in {1,2,3,4}
\path[fill=red,draw=black] (\i,0) circle (2pt) node[below,black] {$\scriptstyle\i$};
\fill[fill=red,draw=black] (7,0) circle (2pt) node[below,black] {$\scriptstyle 2n-1\ $};
\fill[fill=red,draw=black] (8,0) circle (2pt) node[below,black] {$\ \scriptstyle 2n$};
\end{tikzpicture}}
\]
Then $\B_n \simeq \S_{2n}/\H_n$, where $\H_n$ is the stabilizer of $\beta_n$.
Explicitly, $\H_n$ is the subgroup of $\S_{2n}$
that is generated by the transpositions $\transp{2i-1}{2i}$,
$i=1,\ldots,n$, and by the ``double elementary transpositions''
$\transp{2i-1}{2i+1}\transp{2i}{2i+2}$, $i=1,\ldots,n-1$,
making it isomorphic to
the semi-direct product $(\Z/2\Z)^n\ltimes\S_n$ 
(or wreath product $\Z/2\Z\wr\S_n$) 
a.k.a.\ the hyperoctahedral group. 

It is convenient to define $\C[\B_n]$ as the subspace of $\C[\S_{2n}]$
consisting of vectors which are stable by right multiplication by $\H_n$. This
way we can work entirely inside the group algebra $\C[\S_{2n}]$ (this means that we use the definitions
of section \ref{secmurphy} with $n$ replaced with $2n$). Define
\[
P_{\H_n}=\frac{1}{|\H_n|}
\sum_{h\in \H_n} h
\]
$P_{\H_n}$ averages over the group $\H_n$. 
Its image by multiplication 
on the right on $\C[\S_{2n}]$ is exactly $\C[\B_n]$.
$\C[\B_n]$ possesses a standard basis indexed by $\B_n$: to $\pi\in\B_n$ one associates
$\sigma P_{\H_n}$ where $\sigma$ is any permutation such that $\pi=\sigma\beta_n\sigma^{-1}$. 

Below we shall use the property
that in order to average over $\H_n$ 
one can first average over a subgroup $\mathit\Gamma$ of it,
i.e.\ with obvious notations 
$P_{\H_n}=P_{\mathit\Gamma}P_{\H_n}=P_{\H_n}P_{\mathit\Gamma}$.

We now have the following remarkable identity, similar to Prop.~\ref{uprop}:
\begin{prop}\label{keyoprop}
There exists a choice of representatives of cosets of $\S_{2n}/\H_n$,
that is of $\sigma_\pi$ satisfying 
$\sigma_\pi \beta_n \sigma_\pi^{-1}=\pi$ for each $\pi\in\B_n$, such that
\begin{equation}\label{oid}
\prod_{k=1}^n (\tau+m_{2k-1})
=\sum_{\pi\in \B_n} \sigma_\pi\ \tau^{\frac{1}{2}\text{number of cycles of $\beta_n\pi$}}
\end{equation}
\end{prop}
\begin{proof}The proof is extremely similar to that of Prop.~\ref{uprop}.
It is convenient to order the product in the l.h.s.\ as
$(\tau+m_{2n-1})\cdots(\tau+m_1)$.
Note that there are as many terms in the two sides of the equation,
namely $(2n-1)!!$.
Once again we shall provide a term by term identification,
and proceed by induction. Start with a pairing $\pi\in\B_n$. There
are two possibilities. Either (i) $\pi(2n)=2n-1$, in which case
one applies the induction hypothesis to $\pi_{|\{1,\ldots,2n-2\}}$, and
note that $\pi\beta_n$ has two more cycles than their restriction
to $\{1,\ldots,2n-2\}$, namely the two cycles coming from the extra loop
$\vcenterbox{\begin{tikzpicture}[x=0.5cm,y=0.5cm]\path (0,-0.5) rectangle (2,0.5); \draw (0,0) -- (2,0); \draw[thick] (0.5,0) to[in=90,out=90] (1.5,0); \draw[thick] (0.5,0) to[in=270,out=270] (1.5,0); \fill[fill=red,draw=black] (0.5,0) circle (2pt) node[above left=-2pt,black] {$\scriptstyle 2n-1$};
\fill[fill=red,draw=black] (1.5,0) circle (2pt) node[above right=-2pt,black] {$\scriptstyle 2n$};
\end{tikzpicture}}$. Then $\pi$ corresponds to the same term
in $(\tau+m_{2n-3})\cdots(\tau+m_1)$ and to the term $\tau$ in $\tau+m_{2n-1}$,
with the same representative $\sigma_\pi=\sigma_{\pi_{|\{1,\ldots,2n-2\}}}$.
Or (ii) $\pi(2n)<2n-1$, in which case one conjugates $\pi$ by $s_{\pi(2n),2n-1}$,
with the following effect:
\[
\pi=\vcenterbox{
\begin{tikzpicture}[x=0.5cm,y=0.5cm]
\path (0,-1) rectangle (7,1);
\draw (0,0) -- (1.5,0);
\draw[dotted] (1.5,0) -- (2.5,0);
\draw (2.5,0) -- (3.5,0);
\draw[dotted] (3.5,0) -- (4.5,0);
\draw (4.5,0) -- (7,0);
\draw[thick] (1,0) to[in=90,out=90] (6,0);
\draw[thick] (3,0) to[in=90,out=90] (5,0);
\fill[fill=red,draw=black] (1,0) circle (2pt) node[below,black] {$\scriptstyle \pi(2n)$};
\fill[fill=red,draw=black] (3,0) circle (2pt) node[below,black] {$\scriptstyle \pi(2n-1)$};
\fill[fill=red,draw=black] (5,0) circle (2pt) node[below,black] {$\scriptstyle 2n-1$};
\fill[fill=red,draw=black] (6,0) circle (2pt) node[below,black] {$\scriptstyle\ 2n$};
\end{tikzpicture}}
\quad
\transp{\pi(2n)}{2n-1}
\pi
\transp{\pi(2n)}{2n-1}
=\vcenterbox{
\begin{tikzpicture}[x=0.5cm,y=0.5cm]
\path (0,-1) rectangle (7,1);
\draw (0,0) -- (1.5,0);
\draw[dotted] (1.5,0) -- (2.5,0);
\draw (2.5,0) -- (3.5,0);
\draw[dotted] (3.5,0) -- (4.5,0);
\draw (4.5,0) -- (7,0);
\draw[thick] (1,0) to[in=90,out=90] (3,0);
\draw[thick] (5,0) to[in=90,out=90] (6,0);
\fill[fill=red,draw=black] (1,0) circle (2pt) node[below,black] {$\scriptstyle \pi(2n)$};
\fill[fill=red,draw=black] (3,0) circle (2pt) node[below,black] {$\scriptstyle \pi(2n-1)$};
\fill[fill=red,draw=black] (5,0) circle (2pt) node[below,black] {$\scriptstyle 2n-1$};
\fill[fill=red,draw=black] (6,0) circle (2pt) node[below,black] {$\scriptstyle\ 2n$};
\end{tikzpicture}}
\]
Now apply the induction hypothesis to $\pi':=\transp{\pi(2n)}{2n-1}
\pi
\transp{\pi(2n)}{2n-1}_{|\{1,\ldots,2n-2\}}$, noting that the number of cycles
of $\pi'\beta_{n-1}$ is the same as that of $\pi\beta_n$ (the loop
passing through $\pi(2n-1)$ and $\pi(2n)$ has simply shrunk).
Define $\sigma_\pi=\transp{\pi(2n)}{2n-1}\sigma_{\pi'}$. Then
$\sigma_\pi \beta_n \sigma_\pi^{-1}=\pi$, and $\pi$ corresponds to the same
term in $(\tau+m_{2n-3})\cdots(\tau+m_1)$ as $\pi'$, and to
the transposition $\transp{\pi(2n)}{2n-1}$ in $\tau+m_{2n-1}$.

The statement at $n=1$ is trivial ($\sigma_{\beta_1}=1$).
\end{proof}

Note that contrary to the unitary case, the l.h.s.\ of \eqref{oid} is a nonsymmetric 
polynomial of the Jucys--Murphy elements and
is therefore not central.
Its introduction is justified by the
\begin{lemma}Let $G$ be the quantity in \eqref{oid}:
\begin{equation}\label{defGO}
G=\prod_{k=1}^n (\tau+m_{2k-1})
\end{equation}
Then $G P_{\H_n}=P_{\H_n} G$, that is $G$ acting by multiplication on the right
leaves the subspace $\C[\B_n]$ stable.
Furthermore, its matrix in the standard basis of $\C[\B_n]$ is $\mathbf{G}$.
\end{lemma}
\begin{proof}
Proposition \ref{keyoprop} implies that
\[
GP_{\H_n}=\frac{1}{|\H_n|}
\sum_{\sigma\in\S_{2n}} \sigma\ \tau^{\frac{1}{2}\text{number of cycles of $\sigma\beta_n\sigma^{-1}\beta_n$}}
\]
Apply to both sides the anti-isomorphism of 
$\C[\S_{2n}]$ which sends a permutation to its inverse. Clearly the r.h.s.\ 
is invariant, as
are $G$ and $P_{\H_n}$; thus $GP_{\H_n}=P_{\H_n}G$. Since $\C[\B_n]$ is 
by definition the image of $P_{\H_n}$ acting on the right, we have the
first part of the lemma.

The second part consists in rewriting the equality above (with $GP_{\H_n}$ replaced
with $P_{\H_n}G$) in the standard
basis of $\C[\B_n]$, the $\sigma_\pi P_{\H_n}$ ($\pi\in\B_n$):
\begin{align*}
\sigma_\pi P_{\H_n} G &=
\frac{1}{|\H_n|}
\sum_{\sigma\in\S_{2n}} \sigma_\pi\sigma\ \tau^{\frac{1}{2}\text{number of cycles of $\sigma\beta_n\sigma^{-1}\beta_n$}}\\
&=
\frac{1}{|\H_n|}
\sum_{\sigma'\in\S_{2n}} \sigma'\ \tau^{\frac{1}{2}\text{number of cycles of $\sigma'\beta_n\sigma'{}^{-1}\sigma_\pi\beta_n\sigma_\pi^{-1}$}}&& (\sigma'=\sigma_\pi\sigma)\\
&=\sum_{\pi'\in\B_{n}} \mathbf{G}_{\pi',\pi}\, \sigma_{\pi'} P_{\H_n} && (\pi'=\sigma'\beta_n\sigma'{}^{-1})
\end{align*}
\end{proof}

As an amusing corollary, this implies that the matrices
$\mathbf{G}$ commute for distinct values of the parameter $\tau$.

In order to obtain an explicit formula for $\mathbf{G}$, we need the following
\begin{prop}
Let $T$ be a tableau with $2n$ boxes. Then
$P_{\H_n} e_T\ne 0$ implies that the numbers $(2k-1,2k)$ are on the same line of $T$
(i.e.\ $\bigboxes\vcenterbox{\tableau{\scriptscriptstyle 2k-1&\scriptscriptstyle 2k\\}}$)
for all $k=1,\ldots,n$.
\end{prop}
This can actually be found in \cite{BG-zonal-domino}.
\begin{proof}
Let $k$ be an integer between $1$ and $n$.
The key of the proof is the following identity:
\begin{equation}\label{keyid}
P_{\H_n}(m_{2k}-m_{2k-1}-1)
=0
\end{equation}
Writing $P_{\H_n}(m_{2k}-m_{2k-1}-1)=P_{\H_n}
(\transp{2k-1}{2k}-1+\sum_{i=1}^{k-1}(\transp{2i-1}{2k}-\transp{2i-1}{2k-1}+
\transp{2i}{2k}-\transp{2i}{2k-1}))$, one notes that the first term
is annihilated by $1+\transp{2k-1}{2k}$, that is averaging over a subgroup
$\Z/2\Z$ of $\H_n$,
while the $i^{\text{th}}$ term in the sum is annihilated
by averaging over the subgroup $\H_2$ of $\H_n$ which acts on $\{2i-1,2i,2k-1,2k\}$
(the latter calculation being simply the special case $n=2$ of the formula),
which results in the equality \eqref{keyid}.

Now multiply on the right \eqref{keyid} by $e_T$ for some standard Young tableau $T$.
We find
\[
P_{\H_n} e_T 
 (c(T_{2k})-c(T_{2k-1})-1)=0
\]
If $P_{\H_n}e_T\ne 0$, then $c(T_{2k})-c(T_{2k-1})-1=0$, which by inspection
implies that the box $2k$ is directly to the right of the box $2k-1$ in $T$.
\end{proof}


As a consequence, the tableaux $T$ for which $P_{\H_n} e_T\ne 0$ have a very special structure; they are in bijection with tableaux of $n$ boxes by the ``doubling'' procedure
in which each box $\vcenterbox{\tableau{k}}$ is replaced with two boxes
$\bigboxes\vcenterbox{\tableau{\scriptscriptstyle 2k-1&\scriptscriptstyle 2k\\}}$.
E.g.
\[
\tableau{1&3\\2\\}\quad\longrightarrow\quad\tableau{1&2&5&6\\3&4\\}
\]
In particular, the only allowed shapes have {\em even}\/ lengths of rows;
let us denote them by $2\lambda$ where $\lambda\vdash n$.

We can now proceed similarly to the calculation 
of the previous section, that is insert  say on the left
$1=\sum_T e_T$ in \eqref{defGO} and then project using
$P_{\H_n}$.  We find that
$P_{\H_n}e_T G$ is either zero or depends on $T$ only via its shape $2\lambda$.
We conclude by direct computation that
\[
P_{\H_n}G=P_{\H_n} \sum_{\lambda\vdash n} c_\lambda P_{2\lambda}
\qquad
c_\lambda:=\prod_{(i,j)\in\lambda} (\tau+2j-1-i)
\]
In other words, multiplication on the right by $G$
on $\C[\B_n]$ is the same as multiplying by $\sum_{\lambda\vdash n} 
c_\lambda P_{2\lambda}$ (on the left or the right, since the $P_{2\lambda}$
are central).
Thus, we finally have the
\begin{prop}
Define
\begin{equation}\label{final}
W=\sum_{\substack{\lambda\vdash n\\c_\lambda\ne0}} c_\lambda^{-1} P_{2\lambda} 
\end{equation}
Then $\mathbf{W}$ is the matrix of $W$ acting by multiplication 
on the left or the right on $\C[\B_n]$.
\end{prop}
Note that $P_{2\lambda}$ plays here the role of projector onto
isotypic components of $\C[\B_n]$ as a $\C[\S_{2n}]$-module.
The main result of \cite{CM-WOn} (Thm 3.1) is formula \eqref{final} in which
these projectors have been rewritten more explicitly
using \eqref{char}, i.e.\ 
$(P_{2\lambda})_{\pi,\pi'}=\frac{\chi_{2\lambda}(1)}{|\S_{2n}|}\sum_{\sigma\in\S_{2n}: \sigma\pi'=\pi\sigma}
\chi_{2\lambda}(\sigma^{-1})$.


\renewcommand\MR[1]{\relax\ifhmode\unskip\spacefactor3000 \space\fi
  \MRhref{#1}{{\sc mr}}}
\renewcommand{\MRhref}[2]{%
  \href{http://www.ams.org/mathscinet-getitem?mr=#1}{#2}}
\bibliography{../biblio}
\bibliographystyle{amsplainhyper}

\end{document}